\documentclass[11pt]{amsart}
\usepackage[english]{babel}
\usepackage{amsfonts}
\usepackage{amssymb}
\usepackage{amsmath}
\usepackage{amsmath}
\usepackage{cases}
\usepackage[latin1]{inputenc}
\usepackage{mathrsfs}
\usepackage{dsfont}
\usepackage{fancyhdr}
\usepackage{enumerate}  
\usepackage{a4wide}
\usepackage{color}
\usepackage{setspace}

\newtheorem{theorem}{Theorem}[section]

\newtheorem{corollary}[theorem]{Corollary}
\newtheorem{remark}[theorem]{Remark}

\newcommand{\yb}{\mathbf{y}}

\newcommand{\ud}{\mathrm{d}}

\newcommand{\ind}{1 \! \textrm{l}}

\def\reals{\mathbb{R}}

\newcommand{\fptheta}{f(\cdot\ |\ \theta)}
\newcommand{\empiric}{\mathfrak{e}_n}

\def\ps{\Theta}
\def\reals{\mathbb{R}}

\def\ss{\mathbb{X}}
\def\samplespace{\mathbb{X}}
\def\ssa{\mathscr{X}}
\def\psa{\mathscr{T}}

\def\ud{\mathrm{d}}

\def\yb{\mathbf{y}}
\def\tb{\mathbf{t}}

\def\rk{\mathbb{R}^k}

\def\reals{\mathbb{R}}

\def\ind{\mathds{1}}

\title{On uniform continuity of posterior distributions 
}

 \author{{Emanuele} {Dolera} and Edoardo Mainini
 }

 \subjclass[2010]{62F15}
 \keywords{Bayes theorem, Bayesian well-posedness, Bayesian consistency, Continuous dependence on data, exponential models}

 \address[Emanuele Dolera]{Dipartimento di Matematica ``F. Casorati'', Universit\`a di Pavia, via Ferrata 5, I-27100 Pavia, Italy
 }
 \email{emanuele.dolera@unipv.it}
\date{}
\address[Edoardo Mainini]{Dipartimento di ingegneria meccanica, energetica, gestionale e dei trasporti, Universit\`a di Genova, via all'Opera Pia 15, I-16145 Genova,  Italy 
}
\email{mainini@dime.unige.it}

\begin{document}
\maketitle

\begin{abstract}
In the setting of dominated statistical models, we provide conditions yielding strong continuity  of the posterior distribution with respect to the observed data. We show some applications, with special focus on exponential models.  
\end{abstract}

\section{Introduction} \label{sect:Intro}

We investigate the notion of \emph{well-posedness} of a Bayesian statistical inference. For a given \emph{conditional probability distribution}, we refer to well-posedness as a continuity property with respect to the conditioning variable. Indeed,  
we aim at quantitative estimates of the discrepancy between two inferences in terms of the distance between the observations. 
Our problem could be compared with the Bayesian sensitivity 
analysis by specifying that we are working under fixed prior and statistical model, the imprecision being concerned only with the data. 

Few general results are available on this topic, even if it naturally arises---sometimes as a technical tool---in connection with different Bayesian procedures, such as consistency \cite{DF}, \cite[Chapters 6-9]{GV}, mixture approximations \cite{W,RS}, deconvolution \cite{E}, inverse problems \cite{S1} and computability \cite{AFR}. While the pioneering paper \cite{Z1}, essentially inspired by foundational questions, dealt with the qualitative definition of continuity for conditional distributions, more recent studies highlight  the relevance of modulus of continuity estimates. We refer, for instance, to the well-posedness theory developed in \cite{S1} and to different results found in \cite{DS,CDRS,ILS,L}. 
Our contribution moves in the same direction.

In this work, we confine ourselves to dealing with the \emph{posterior distribution}. We introduce two measurable spaces $(\ss, \ssa)$ and $(\ps, \psa)$, representing the space of the observations and the parameters, respectively.  We further introduce a probability measure $\pi$ on $(\ps, \psa)$, the \emph{prior distribution}, and  a probability kernel $\nu(\cdot \  |\ \cdot) : \ssa \times \ps \rightarrow [0,1]$, called \emph{statistical model}. We assume that: 
\begin{itemize}
\item $\ss$ is a metric space with distance $d_\ss$  and $\ssa$ coincides with the Borel $\sigma$-algebra on $\ss$; 
\item $\ps$ is a Polish  space and $\psa$ coincides with its Borel $\sigma$-algebra;
\item the model is \emph{dominated}: $\forall\ \theta \in \Theta$, $\nu(\cdot\ |\ \theta) \ll \lambda$, for some $\sigma$-finite measure $\lambda$  on $(\ss,\ssa)$.
\end{itemize}
For any $\theta \in \ps$, we consider a density $\fptheta$ of $\nu(\cdot\ |\ \theta)$ w.r.t. $\lambda$ and we put $\rho(x) :=\int_\ps f(x\ |\ \theta)\,\pi(\ud\theta)$, since
$(x,\theta) \mapsto f(x\ |\ \theta)$ proves to be $\ssa \otimes \psa$-measurable. See \cite[Chapter 5]{K} for details. In this framework, the well-known \emph{Bayes theorem} provides an explicit form of the posterior distribution, namely 
$$
\pi(B\ |\ x) =  \dfrac{\int_B f(x\ |\ \theta) \pi(\ud\theta)}{\rho(x)}
$$ 
for any $B \in \psa$ and $x$ such that $\rho(x)>0$. Thus, the Bayes mapping $x \mapsto \pi(\cdot\ |\ x)$ can be seen as a measurable function from $\{x\in \ss\ |\ \rho(x)>0\}$ into the space  of all probability measures on $(\ps, \psa)$ endowed with the topology of weak convergence. Our main task is to find sufficient conditions on $\pi$ and $f(\cdot\ |\ \cdot)$ such that $x \mapsto \pi(\cdot\ |\ x)$ satisfies a uniform continuity condition of the following form: given a modulus of continuity $w:[0,+\infty)\to[0,+\infty)$ and a set $K\subseteq\{x\ |\ \rho(x)>0\}$, there exists a constant $L_w(K)$ such that
\begin{equation} \label{eq:main}
\ud_{TV}(\pi(\cdot\ |\ x),\pi(\cdot\ |\ y)) \leq L_{w}(K) \,w(\ud_\ss(x,y)) \quad\quad \forall\ x,y \in K
\end{equation}
holds, where $\ud_{TV}(\pi_1, \pi_2) := \sup_{B \in \psa} |\pi_1(B) - \pi_2(B)|$ denotes the \emph{total variation distance}.

In order to motivate the study of a property like \eqref{eq:main}, let us briefly discuss some of its applications to Bayesian inference. By itself, uniform continuity is of interest in the theory of regular conditional distributions \cite[Sections 9.6-9]{T}, \cite{P}. Indeed, a natural approximation of the posterior is
$$
\overline{\pi}_x(B) := \frac{\int_{B\times U_x} f(y\ |\ \theta) \pi(\ud\theta) \lambda(\ud y)} {\int_{\ps\times U_x} f(y\ |\ \theta) \pi(\ud\theta) \lambda(\ud y)},\qquad B\in\psa,
$$
where $U_x$ stands for a suitable neighborhood of $x$. Thus, \eqref{eq:main} would express the approximation error $\ud_{TV}(\pi(\cdot\ |\ x), \overline{\pi}_x)$. 
Anyway, the main applications are concerned with the theory of $n$ exchangeable observations, where $x = (x_1, \dots, x_n)$, $y = (y_1, \dots, y_n)$ and the model
$\fptheta$ is in product form, by de Finetti's representation theorem. The main advantage of an estimate like \eqref{eq:main} arises when $\ud_\ss(x,y)$ is re-expressed in terms of a sufficient statistic (e.g., the empirical measure), so that the asymptotic behavior of the posterior for large $n$ could be studied by resorting to the asymptotic behavior of such statistic. We believe that uniform continuity would represent a new technique to prove asymptotic properties of the posterior distribution, like Bayesian consistency. See Sections 3.2 and \ref{sect:consistency} below. Finally, uniform continuity would represent also a powerful technical tool to solve the problem
of approximating the posterior by mixtures, on the basis of a discretization of the sample space. See \cite{RS} and, in particular, Proposition 2 therein, where 
an estimate like \eqref{eq:main} would
allow to quantitatively determine how fine the discretization should be in order to achieve a desired degree of approximation.

%

\section{Continuous dependence on data}
In the sequel, we refer to a \emph{modulus of continuity} as a continuous strictly increasing function $w:[0,+\infty)\to[0,+\infty)$ such that $w(0) = 0$, and we consider the space  of $w$-continuous functions. In particular, we say that $g:K\subseteq\ss\to\mathbb{R}$ belongs to $ C^w(K)$ if
\[
|g|_{C^w(K)} :=\sup_{x,y\in K\\ \atop{x\neq y}} \dfrac{|g(y)-g(x)|}{w(d_\ss(x,y))} < +\infty.
\] 
If $w(r)=r^\alpha$, $\alpha\in(0,1]$, we get the class $C^{0,\alpha}(K)$ of H\"older continuous functions.

\begin{theorem} \label{thm:main} 
In the same setting of \emph{Section \ref{sect:Intro}}, suppose that $R(K):=\inf_K\rho>0$ is fulfilled for some $K\subseteq \ss$ and that, for a suitable modulus of continuity $w$, there holds
\begin{equation}\label{hypothesis}
A_{w,f,\pi,K}:=\int_{\Theta} |\fptheta|_{C^{w}(K)} \pi(\ud \theta) < +\infty.
\end{equation}
Then, \eqref{eq:main} is satisfied with $L_{w}(K) =\displaystyle\frac{A_{w,f,\pi,K}}{R(K)}$.
\end{theorem}

\begin{proof} 
First of all, by assumption \eqref{hypothesis}, and since $\rho(x) =\int_\ps f(x\ |\ \theta)\,\pi(\ud\theta)$,  there holds
\begin{equation}\label{af2}\begin{aligned}
|\rho(y)-\rho(x)|
\le\int_\Theta |f(y\ |\ \theta)-f(x\ |\ \theta)|\,\pi(\ud\theta)\le A_{w,f,\pi,K} \,w(\ud_\ss(y,x))
\end{aligned}\end{equation}
for any $x,y \in K$, so that $\rho\in C^w(K)$. The dual formulation of the total variation (see, e.g., \cite{GibbsSu}) reads
$$
\ud_{TV}(\pi(\cdot\ |\ x),\pi(\cdot\ |\ y))=\frac{1}{2} \sup_{|\zeta|\le 1}\left(\int_\Theta \zeta(\theta)\,\pi(\ud\theta\ |\ x)-\int_\Theta \zeta(\theta)\,\pi(\ud\theta\ |\ y)\right)
$$
for any $x,y\in\ss$, the supremum being taken among all continuous functions $\zeta:\Theta\to\mathbb{R}$ such that $|\zeta(\theta)| \le 1$ for any $\theta \in \ps$. For any such $\zeta$, define
 $\Phi_\zeta(x):=\int_\Theta \zeta(\theta)\,\pi(\ud\theta\ |\ x)$ and note that the Bayes formula entails $\rho(x)\Phi_\zeta(x)=\int_{\Theta}\zeta(\theta)f(x\ |\ \theta)\,\pi(\ud\theta)$.
We shall prove the  $w$-continuity of the map $x\mapsto \Phi_\zeta(x)$ on $K$. First of all, this map satisfies $|\Phi_\zeta(x)|\le 1$ for any $x \in \ss$ since $|\zeta(\theta)| \le 1$.
Then, for $x,y\in K$, there holds
\begin{equation}\label{af1}
|\rho(y)\Phi_\zeta(y)-\rho(x)\Phi_\zeta(x)| = \left|\int_\Theta\zeta(\theta)(f(y \ | \ \theta)-f(x\ |\ \theta))\,\pi(\ud\theta)\right| \le A_{w,f,\pi,K}\, w(\ud_\ss(y,x)),
\end{equation}
yielding $\rho\Phi_\zeta \in C^w(K)$. Since $\rho\ge R(K)>0$ on $K$, and $|\Phi_\zeta|\le 1$, for any $x,y\in K$ we get
\begin{equation}\label{diff}
R(K)|\Phi_\zeta(y)-\Phi_\zeta(x)| \le |\rho(y)\Phi_\zeta(y)-\rho(y)\Phi_\zeta(x)|
\le |\rho(y)\Phi_\zeta(y)-\rho(x)\Phi_\zeta(x)|+|\rho(y)-\rho(x)|.
\end{equation}
For any $x,y\in K$ such that $x\neq y$, \eqref{af2}--\eqref{diff} entail
\[
R(K)\,\frac{|\Phi_\zeta(y)-\Phi_\zeta(x)|}{w(\ud_\ss(y,x))}\le \frac{|\rho(y)\Phi_\zeta(y)-\rho(x)\Phi_\zeta(x)|}{w(\ud_\ss(y,x))}+\frac{|\rho(y)-\rho(x)|}{w(\ud_\ss(y,x))}\le 2A_{w,f,\pi,K}.
\]
The latter estimate being uniform with respect to $\zeta$, we conclude that
\[
\frac{\ud_{TV}(\pi(\cdot\ |\ y),\pi(\cdot\ |\ x))}{w(\ud_\ss(y,x))}= \frac{1}{2} \sup_{|\zeta|\le 1}
\frac{|\Phi_\zeta(y)-\Phi_\zeta(x)|}{w(\ud_\ss(y,x))}\le \frac{A_{w,f,\pi,K}}{R(K)}
\]
holds for any $x, y\in K$ such thay $x\neq y$, proving the theorem. 
\end{proof}

\begin{remark}\label{>0}\rm
If $\lambda(\ss)<+\infty$, we can take $K=\ss$ in Theorem \ref{thm:main}. If $R(\ss)>0$, we get $w$-continuity on the whole  $\ss$ for the map $x\mapsto \pi(\cdot\ |\ x)$, w.r.t. $\ud_{TV}$.
\end{remark}

Some examples may also be treated within the following simple

\begin{corollary}\label{coro1} 
In the same framework of {\rm Theorem \ref{thm:main}}, take $K\subseteq\samplespace$ such that \eqref{hypothesis} holds. In addition,
suppose there exist $g:\Theta\to \mathbb{R}$ and $h:\samplespace\to\mathbb{R}$ such that  $g>0$ on $\Theta$, $\inf_K h>0$ and $f(x\ |\ \theta)\ge g(\theta)h(x)$ for any $\theta\in\Theta$ and $x\in K$. Then, \eqref{eq:main} is satisfied with $L_{w}(K) =\displaystyle A_{w,f,\pi,K} \cdot \left[ \inf_{x\in K} h(x) \,\int_\Theta g(\theta)\,\pi(\ud\theta) \right]^{-1}$.
\end{corollary}

Let us now consider the Euclidean case, letting $\ss\subseteq\reals^d$ have nonempty interior and $K\subseteq \ss$ be an open set with Lipschitz boundary. Usually, 
a Sobolev regularity might be simpler to verify, the H\"older regularity following then by Sobolev embedding. For instance, for $p>d$, by Morrey inequality there exists a constant 
$C_{1,d,p}(K)$ such that $|g|_{C^{0,\alpha}(K)}\le C_{1,d,p}(K)| g|_{W^{1,p}({K})}$ holds for any $g\in W^{1,p}(K)$, with $\alpha=1-d/p$ and $|g|_{W^{1,p}(K)}:=\|\nabla g\|_{L^{p}({K})}$. 
More generally, if $1>s>d/p$, the fractional Sobolev embedding (see, e.g., \cite{DD}) states that 
\begin{equation} \label{Sobolev}
|g|_{C^{0,\alpha}(K)}\le C_{s,d,p}(K)|g|_{W^{s,p}({K})}
\end{equation}
holds with a suitable constant $C_{s,d,p}(K)$, $\alpha=s-d/p$ and 
$$
|g|_{W^{s,p}({K})}:=\left(\int_{{K}}\int_{{K}}\frac{|g(x)-g(y)|^p}{|x-y|^{d+sp}}\,\ud x\,\ud y\right)^{\frac1p}<+\infty
$$ 
for any $g\in L^p(K)$. We readily obtain the following

\begin{corollary} 
In the same framework of Theorem {\rm \ref{thm:main}}, let $d/p<s\le 1$ and let ${K\subseteq\ss\subseteq\reals^d}$ be an open set with Lipschitz boundary. Let
$B_{p,s,f,\pi,K}:=\int_{\Theta} |\fptheta|_{W^{s,p}(K)} \pi(\ud \theta) < +\infty$ and $R(K):=\inf_K\rho>0$. Then, for $\alpha=s-d/p$ and $w(r) = r^{\alpha}$, \eqref{eq:main} is satisfied with $L_{w}(K) = \displaystyle\frac{C_{s,d,p}(K)\,B_{p,s,f,\pi,K}}{R(K)}$.
 \end{corollary}

\section{Examples and applications}

\subsection{Exponential models}

A remarkably interesting statistical model is the \emph{exponential family}, which includes many popular distributions, such as the Gaussian, the exponential and the gamma. For terminology and basic results about this family, see, e.g., \cite{brown}. For the sake of definiteness, we consider a $\sigma$-finite reference measure $\lambda$ on $(\ss, \ssa)$ and a measurable map $\tb : \ss \rightarrow \rk$ such that
the interior $\Delta$ of the convex hull of the support of $\lambda \circ \tb^{-1}$ is nonempty and $\Lambda := \Big\{\yb \in \rk\ \Big|\ \int_{\ss} e^{\yb \cdot \tb(x)} 
\lambda(\ud x) < +\infty\Big\}$ is a nonempty open subset of $\rk$. 
As for $\fptheta$, we resort to the so-called \emph{canonical parametrization}, by which $\ps = \Lambda$, $\theta = \yb$,
$$f(x\ |\ \theta) = e^{\theta \cdot \tb(x) - M(\theta)},\quad M(\theta) := \log\left(\int_{\ss} e^{\theta \cdot \tb(x)} \lambda(\ud x)\right)$$ and, for any $\theta \in \ps$, 
$\fptheta$ is a probability density function w.r.t. $\lambda$. Now, given a prior $\pi$ on $(\ps,\psa)$ and a set $K$ compactly contained in the interior of $\ss$, we observe that
$$R(K):=\inf_{x\in K}\rho(x)\ge \int_\Theta \inf_{x\in K} f(x\ |\ \theta)\,\pi(\ud\theta) = \int_{\Theta}e^{\inf_{x\in K}\theta\cdot \tb(x)}e^{-M(\theta)}\,\pi(\ud\theta),$$
where the last term is positive if $\tb$ is continuous. On the other hand, we have $$\int_{\Theta} |f(\cdot\ |\ \theta)|_{C^w(K)}\,\pi(\ud\theta) \le |\tb|_{C^w(K)} \int_\Theta  |\theta| e^{\sup_{z\in t(K)}\theta\cdot\yb}\, e^{-M(\theta)}\,\pi(\ud\theta).$$ Therefore, if we suppose $\tb\in C^w(K)$ and that the integral $\int_\Theta  |\theta| e^{\sup_{z\in t(K)}\theta\cdot\yb}\, e^{-M(\theta)}\,\pi(\ud\theta)$ is finite, we can invoke Theorem \ref{thm:main} to obtain the  $w$-continuity on $K$ of the posterior distribution.

We finally notice that, in connection with an exponential model, it is natural to choose a \emph{conjugate prior}, yielding an explicit form of the posterior \cite{DY}. Thus, the LHS of \eqref{eq:main} can be directly computed, claiming a fair comparison with the RHS. Actually, nothing seems lost at the level of the modulus of continuity, though our constant $L_w(K)$ is usually sub-optimal. To illustrate this phenomenon, we can take $f(x\ |\ \theta) = \theta e^{-\theta x}$, with $x \in \ss = [0,+\infty)$ and $\theta \in \Theta = (0, +\infty)$. Chosen a conjugate prior like $\pi(\ud\theta) = e^{-\theta}\ud\theta$, we observe that \eqref{eq:main} holds with $K = [0,M]$, for any $M>0$, and $w(r)=r$. But our constant $L_w(K)$ behaves asymptotically like $M^2$ for large $M$, whilst the optimal one remains bounded as $M$ grows.

\subsection{Exponential models for $n$ exchangeable observations} \label{sect:nexchangeable}

Here, we adapt the result of Section 3.1 to the $n$-observations setting, assuming exchangeability. In this case $x=(x_1,\ldots x_n)\in\mathbb{X}^n$ and the statistical model is of the form
$(\ss^n,\ps)\ni (x,\theta)\mapsto f(x\ | \ \theta)=\prod_{i=1}^{n} \tilde f(x_i \ | \ \theta)$
for some density $ \tilde f(\cdot \ | \ \cdot): \ss\times \ps\to \mathbb{R}$. If $\tilde f$ belongs to the exponential family considered in Section 3.1, we have
$\prod_{i=1}^n \tilde f(x_i\ |\ \theta) = \exp\big\{\theta \cdot \big(\sum_{i=1}^n \tb(x_i)\big) - nM(\theta)\big\}.$ By Neyman's factorization lemma, we rewrite the model as $$f(x\ |\ \theta)= \exp\{\tau_n(\theta) \cdot \bar{\tb}_n(X_n(x)) - \bar M_n(\tau_n(\theta))\},$$
for suitable functions $X_n:\mathbb{X}^n\to\ss^k$, $\tau_n:\Theta\to \mathbb{R}^k$, $\bar{\tb}_n:\ss^k\to\mathbb{R}^k$, $\bar M_n:\mathbb{R}^k\to\mathbb{R}$, with $X_n$ symmetric
and $k$ standing for the dimension of $\Theta$. We can recast the statistical model by considering $\tau_n$ as the new parameter, and $X_n$ as the observable. Indeed, we introduce
$$g(X,\tau)=\exp\{\tau \cdot \bar{\tb}_n(X)) - \bar Q_n(\tau)\}\,\exp\{\bar Q_n(\tau) - \bar M_n(\tau)\},$$
where $e^{\bar Q_n(\tau)}:=\int_{\ss^k} e^{\tau\cdot \bar{\tb}_n(Z)}\lambda^{\otimes k}(dZ)$ and $\lambda$ is the reference measure on $\ss$.
Letting $\varphi_n(\tau):=\exp\{\bar Q_n(\tau) - \bar M_n(\tau)\}$, we have $g(X,\tau)=h(X\ | \ \tau)\varphi_n(\tau)$, where $h(\cdot \ | \ \tau)$ is a probability density with respect to the product measure $\lambda^{\otimes k}$, parametrized by $\tau\in \mathbb{R}^k$. Given a prior $\pi$ on $(\ps,\psa)$, the posterior $\pi_n(\ud\theta\ | \ x)$ reads
\[
\dfrac{\prod_{i=1}^{n} \tilde f(x_i \ | \ \theta)\,\pi(\ud\theta)}{\int_\Theta \prod_{i=1}^n\tilde f(x_i\ | \ t)\pi(\ud t)}
=\dfrac{h(X_n(x) \ | \ \tau_n(\theta))\,\varphi_n(\tau_n(\theta))\,\pi(\ud\theta)}{\int_\Theta h(X_n(x)\ | \ \tau_n(t))\,\varphi_n(\tau_n(t))\,\pi(\ud t)}=
\dfrac{h(X_n(x) \ | \ \tau_n(\theta))\,\bar\pi_n(\ud\theta)}{\int_\Theta h(X_n(x)\ | \ \tau_n(t))\,\bar\pi_n(dt)},
\]
where $\bar\pi_n(\ud\theta):=\displaystyle\frac{\varphi_n(\tau_n(\theta))\,\pi(\ud\theta)}{\int_\Theta \varphi_n(\tau_n(t))\,\pi(\ud t)}$, provided that the denominator is finite. Then, if Theorem \ref{thm:main} can be applied in terms of the new model $h$ and prior $\bar \pi_n$, the thesis \eqref{eq:main} reads
\[
\ud_{TV}(\pi_n(\cdot \ | \ x),\pi_n(\cdot \ | \ y))\le L_w(K)\, w(\ud_{\ss^k}(X_n(x),X_n(y)))\qquad \forall x,y\in K,
\]
 where $\ud_{\ss^k}$ denotes the product distance. We stress that a bound in terms of  $\ud_{\ss^k}(X_n(x),X_n(y))$ is statistically more meaningful than a bound in terms of $d_{\ss^n}(x,y)$, as the former agrees with the symmetry assumption coming from exchangeability.

\subsection{Global regularity for models with $\ss \subset \mathbb R^d$.}

When $\lambda(\ss) < +\infty$, we can check whether Theorem \ref{thm:main} holds with $K = \ss$, yielding a global uniform continuity. We discuss the case $\ss \subset \mathbb R^d$ with $\lambda = \mathscr{L}^d$, the $d$-dimensional Lebesgue measure, forcing $\ss$ to be bounded. Many popular models do not satisfy the assumptions of Theorem \ref{thm:main} with $K = \ss$, but only with some $K$ compactly contained in the interior of $\ss$. This is the case of Beta and Dirichlet models. On the other hand, a model that fits the assumptions of 
Corollary \ref{coro1} is the Bradford distribution, given by 
$$
f(x\ |\ \theta):=\frac{\theta}{(1+\theta x)\,\log(1+\theta)}
$$ 
with $x\in\samplespace:= [0,1]$ and $\theta \in \Theta:=(-1,+\infty)$. Such a model is used for the description of the occurencies of references in a set of documents on the same subject \cite{L1}. Choosing 
\[
g(\theta):=\left\{
\begin{array}{ll}\medskip
\dfrac{\theta}{\log(1+\theta)}\quad&\mbox{ if $\theta\in(-1,0]$}\\
\dfrac{\theta}{(1+\theta)\log(1+\theta)}\quad&\mbox{ if $\theta\in(0,+\infty)$}\ ,
\end{array}
\right.
\]
Corollary \ref{coro1} entails global Lipschitz-continuity, i.e.\! \eqref{eq:main} with $L_w(\ss) = \displaystyle\frac{\int_{\Theta} C_1(\theta)\,\pi(\ud\theta)}{\int_\Theta g(\theta)\,\pi(\ud\theta)}$ and $w(r)=r$, where $C_1(\theta):=\sup_{\substack{x\in\samplespace}}|\partial_xf(x\ |\ \theta)|$, provided that  $\pi$ satisfies
$\int_\Theta C_1(\theta)\,\pi(\ud\theta)<+\infty$.

\subsection{Infinite-dimensional models}

One of the merits of our approach consists in the fact that we can handle also complex statistical models of non parametric type. Two noteworthy examples in Bayesian analysis are the \emph{infinite dimensional exponential family} and the \emph{infinite mixture models}. See \cite{GN} for a comprehensive treatment.  

As for the first model, keeping in mind the Karhunen-Lo\'eve theorem \cite[Chapter 13]{K}, we confine ourselves to considering densities of the form
$f(x\ |\ \theta) = e^{\theta(x)} \left(\int_0^T e^{\theta(y)} \ud y\right)^{-1}$
with $\ss = [0,T]$, with a fixed $T > 0$ and $\ps = \mathrm{C}([0,T];\reals)$. See also \cite{L3}. After fixing a prior $\pi$, we show how Theorem \ref{thm:main} can be applied.
First, we deal with the condition on the infimum of $\rho$ with $K = \ss$. In fact, we have $\inf_{x \in \ss} f(x\ |\ \theta) \geq \frac{1}{T} e^{-r_T(\theta)}$, 
where $r_T(\theta) := \sup_{x \in \ss} \theta(x) - \inf_{x \in \ss} \theta(x)$ denotes the \emph{range} of the (random) trajectory $\theta$. Whence,
$$
R(\ss) \geq \frac{1}{T} \int_{\ps} e^{-r_T(\theta)} \pi(\ud\theta)=\frac1T \int_{0}^{+\infty} e^{-s}F_{r_T}(s)\,ds>0\ ,
$$
where $F_{r_T}$ stands for the density of $r_T$ with respect to the Lebesuge measure. To check \eqref{hypothesis}, we consider a H\"older condition with exponent 
$\gamma \in (0,1)$. Thus, we note that 
$$|\fptheta|_{C^{0,\gamma}(\ss)} \leq \frac{1}{T} e^{r_T(\theta)} |\theta|_{C^{0,\gamma}(\ss)}.$$
By H\"older's inequality, for $p>1$ and $\frac1p+\frac1q=1$, we write 
$$
\int_\Theta|\fptheta|_{C^{0,\gamma}(\ss)}\,\pi(d\theta) \leq \frac1T\left(\int_\Theta|\theta|^p_{C^{0,\gamma}([0,T])}\,\pi(d\theta)\right)^{\frac1 p}\left(\int_0^{+\infty} e^{q s} F_{r_T}(s)\,ds\right)^{\frac1q}\ .
$$
By the fractional Sobolev inequality \eqref{Sobolev}, for $s=\frac{\gamma p+1}{p}$ we have
\begin{equation}\label{long}
\begin{aligned}
\int_\Theta|\theta|^p_{C^{0,\gamma}([0,T])}\,\pi(d\theta) &\le C^p_{s,1,p}(\ss)\int_\Theta \int_0^T\!\!\int_0^T\frac{|\theta(x)-\theta(y)|^p}{|x-y|^{1+sp}}\,\ud x\,\ud y\,\pi(\ud \theta)\\
&=C^p_{s,1,p}(\ss)\!\int_0^T\!\!\int_0^T\!\! \frac{1}{|x-y|^{1+sp}}\left(\int_{\reals^2} |u-v|^p\,h_{x,y}(u,v)\,\ud u\,\ud v\right)\ud x\,\ud y,
\end{aligned}
\end{equation}
where $h_{x,y}$ is the \emph{two-times density} of $\theta$. Typically, the Kolmogorov-Chentsov condition \cite[Chapter 3]{K}
\begin{equation}\label{KC}
\int_{\reals^2} |u-v|^p\,h_{x,y}(u,v)\,du\,dv \le Q(p,\lambda)|x-y|^{1+\lambda}
\end{equation}
holds for some $\lambda>0$ and some $Q(p,\lambda)>0$. See, e.g., \cite{K}. If \eqref{KC} is verified for $\lambda>sp-1=\gamma p$, then the last term of \eqref{long} is finite. Summing up, if there exists $q>1$ such that $$Z_q(r_T):=\left(\int_0^{+\infty} e^{q s}F_{r_T}(s)\,\ud s\right)^{\frac 1q}<+\infty,$$ and if \eqref{KC} holds for some $\lambda>0$ and 
$p = q/(q-1)$, then, as soon as $\gamma<\min\{1-\frac1p,\frac\lambda p\}$, \eqref{eq:main} holds with $w(r)=r^\gamma$
and 
$$
A_{w,f,\pi,\ss} = \frac1T\, C_{s,1,p}(\ss)\, Z_q({r_T})\left( Q(p,\lambda)\,\int_0^T\int_0^T|x-y|^{\lambda-\gamma p-1}\,\ud x\,\ud y \right)^{\frac1p}\ .
$$
The case of $\pi$ equal to the Wiener measure deserves some attention. Indeed, \eqref{KC} is satisfied with $\lambda=p/2-1$ and $p>2$ \cite[Chapter 13]{K}.  Moreover, we have (see \cite{F}) 
$$
F_{r_T}(s)=\frac{8}{\sqrt{2\pi}}\sum_{k=1}^{+\infty}(-1)^k k^2 \exp{\left\{-\frac{k^2s^2}{2T}\right\}}\ ,
$$ 
yielding $Z_q(r_T) < +\infty$ for all $q \in \reals$. Therefore, we obtain H\"older continuity of the posterior distribution for any exponent $\gamma<\frac12$.

The second model of interest, namely the so-called infinite dimensional mixture model, is based on a family of densities of the form
$f(x\ |\ \theta) = \int_{\reals} \kappa(x; t) \theta(\ud t)$, with $\ss = \reals$ and $\ps$ equal to the space of all probability measures $\mathcal{P}(\reals)$ on $(\reals, \mathscr{B}(\reals))$. The kernel
$\kappa$ consists of a family of densities (in the $x$-variable) parametrized by $t \in \reals$. A noteworthy case of interest is the Gaussian kernel $\kappa(x; t) = \frac{1}{\sqrt{2\pi}}
\exp\{-\frac 12 (x-t)^2\}$. Now, after fixing a prior $\pi$ of nonparametric type (e.g. the Ferguson-Dirichlet prior), the application of Theorem \ref{thm:main} is straightforward.
First, for kernels in the form $\kappa(x; t) = \kappa(x-t)$, condition \eqref{hypothesis} holds even independently of $\pi$, provided that 
$\sup_{x \in \reals} |\kappa'(x)| < +\infty$. For the condition on $R(K)$, for some compact $K \subset \reals$, it is enough to assume that $\inf_{x \in K} \int_{\reals} \kappa(x; t) \overline{\theta}(\ud t) > 0$, where $\overline{\theta}(B) := \int_{\mathcal{P}(\reals)} \theta(B) \pi(\ud\theta)$, $B \in \mathscr{B}(\reals)$.

\subsection{Application to Bayesian consistency} \label{sect:consistency}

We have seen in Section \ref{sect:nexchangeable} that, in presence of exchangeable observations, the posterior can be written as
$$
\pi_n(\ud\theta\ |\ x) = \frac{\prod_{i=1}^n \tilde{f}(x_i\ |\ \theta)}{\int_{\ps} \prod_{i=1}^n \tilde{f}(x_i\ |\ t) \pi(\ud t)} \pi(\ud\theta) =
\frac{\exp\{ n\int_{\ss} \log\tilde{f}(y\ |\ \theta) \empiric^x(\ud y)\} }{\int_{\ps} \exp\{ n\int_{\ss} \log\tilde{f}(y\ |\ t) \empiric^x(\ud y)\} \pi(\ud t)} \pi(\ud\theta)
$$
where $x = (x_1, \dots, x_n)$ and $\empiric^x(\cdot) := \frac 1n \sum_{i=1}^n \delta_{x_i}(\cdot)$ denotes the \emph{empirical measure}. In the theory of Bayesian consistency, 
one fixes $\theta_0 \in \ps$ and generates from $\nu(\cdot\ |\ \theta_0)$
a sequence $\{\xi_i\}_{i \geq 1}$ of i.i.d. random variables. The objective is to prove that 
the posterior piles up near the true value $\theta_0$, i.e. $\ud\big(\pi(\cdot\ |\ \xi_1, \dots, \xi_n), \delta_{\theta_0}\big) \rightarrow 0$ in probability, for some weak distance 
$\ud$ (e.g., Prokhorov or bounded-Lipschitz metric \cite{GibbsSu}) between probability measures on $(\ps, \psa)$, possibly with an estimation of the convergence rate. To establish a link with our theory, we introduce the probability kernel
$$
\pi_n^*(\ud\theta\ |\ \mu) :=  \frac{\exp\{ n\int_{\ss} \log\tilde{f}(y\ |\ \theta) \mu(\ud y) - M_n(\theta) \}}{\int_{\ps} \exp\{ n\int_{\ss} \log\tilde{f}(y\ |\ t) \mu(\ud y) - M_n(t)\} \pi(\ud t)} 
e^{M_n(\theta) - \overline{M}_n} \pi(\ud\theta)
$$ 
where $\mu \in \mathcal{M}$, a subset of probability measures containing in its closure both $\nu(\cdot\ |\ \theta_0)$ and $\empiric^{(\xi_1, \dots, \xi_n)}(\cdot)$, with 
$$
M_n(\theta) := \log\left\{ \int_{\mathcal{M}} \exp\left\{ n\int_{\ss} \log\tilde{f}(y\ |\ \theta) \mu(\ud y)\right\} \eta(\ud\mu) \right\}
$$ 
for some measure $\eta$ on $\mathcal{M}$, and $\overline{M}_n := \log\left\{ \int_{\ps} e^{M_n(\theta)} \pi(\ud\theta) \right\}$. In this notation, $\pi_n(\ud\theta\ |\ x) = \pi_n^*(\ud\theta\ |\ \empiric^x)$. Whence,
$$
\ud\big(\pi(\cdot\ |\ \xi_1, \dots, \xi_n), \delta_{\theta_0}\big) \leq \ud\big(\pi_n^*(\cdot\ |\ \nu(\cdot\ |\ \theta_0)), \delta_{\theta_0}\big) +
\ud\big( \pi_n^*(\cdot\ |\ \empiric^{(\xi_1, \dots, \xi_n)}), \pi_n^*(\cdot\ |\ \nu(\cdot\ |\ \theta_0)) \big)\ .
$$
As for the first term on the RHS, convergence to zero is well-known with explicit rates, as a consequence of the so-called Kullback-Leibler property \cite[Definition 6.15]{GV}. The second term on the RHS can be studied under expectation, by splitting it as follows:
\begin{equation} \label{expectation}
\mathds{E}_{\theta_0}\left[\ud\big( \pi_n^*(\cdot\ |\ \empiric), \pi_n^*(\cdot\ |\ \nu(\cdot\ |\ \theta_0)) \big) \ind_{E_n} \right] + \mathds{E}_{\theta_0} \left[\ud\big( \pi_n^*(\cdot\ |\ \empiric), \pi_n^*(\cdot\ |\ \nu(\cdot\ |\ \theta_0)) \big) \ind_{E_n^c}\right] 
\end{equation}
with $E_n := \{D(\empiric^{(\xi_1, \dots, \xi_n)},  \nu(\cdot\ |\ \theta_0)) \leq \epsilon_n\}$, where $\{\epsilon_n\}_{n \geq 1}$ is a vanishing sequence of positive numbers and
$D$ a weak distance (e.g., Prokhorov or bounded-Lipschitz metric \cite{GibbsSu}) between probability measures on $(\ss, \ssa)$. If the distance $\ud$ is bounded, the second term in \eqref{expectation} is handled in terms of 
$\mathds{P}_{\theta_0}[D(\empiric,  \nu(\cdot\ |\ \theta_0)) > \epsilon_n]$, and hence resorting to well-known large deviations inequalities for empirical processes \cite[Chapter 27]{K}. Finally, if $\ud \leq \ud_{TV}$ (see \cite{GibbsSu}), we can study the first term in \eqref{expectation} by applying Theorem \ref{thm:main}, with $K = \{\mu \in \mathcal{M}\ |\ 
D(\mu, \nu(\cdot\ |\ \theta_0)) \leq \epsilon_n\}$ and $w(r) = r^{\alpha}$ for some $\alpha \in (0,1]$. The role of the local H\"older continuity is now functional  to reducing the 
analysis of the first term in \eqref{expectation} to that of $\mathds{E}_{\theta_0}\big[ D(\empiric^{(\xi_1, \dots, \xi_n)},  \nu(\cdot\ |\ \theta_0))^{\alpha} \big]$, whose rates of contraction are well-known \cite{FG}.

\section*{Acknowledgements}
ED received funding from the European Research Council (ERC) under the European Union's Horizon 2020 research and innovation programme under grant agreement No 817257. This research was also supported by the Italian Ministry of Education, University and Research (MIUR) under ``PRIN project" grant No 2017TEXA3H, and ``Dipartimenti di Eccellenza Program" (2018--2022) - Dept. of Mathematics ``F. Casorati", University of Pavia.


\end{document}